\documentclass[12pt, reqno]{amsart}
\usepackage[utf8]{inputenc}

\usepackage{latexsym}
\usepackage{amsmath}
\usepackage{amssymb, amscd, amsthm, mathtools}
\usepackage[all]{xy}
\usepackage{multirow}
\usepackage{xcolor} 
\usepackage{tikz}
\usepackage{tikz-cd}
\usepackage{mathrsfs}

\usepackage[shortlabels,inline]{enumitem}
\SetEnumerateShortLabel{a}{\textup{(\alph*)}}
\SetEnumerateShortLabel{A}{\textup{(\Alph*)}}
\SetEnumerateShortLabel{1}{\textup{(\arabic*)}}
\SetEnumerateShortLabel{i}{\textup{(\roman*)}}
\SetEnumerateShortLabel{I}{\textup{(\Roman*)}}

\newcommand{\Q}{\mathbb{Q}}
\newcommand{\Z}{\mathbb{Z}}

\newcommand{\R}{\mathbb{R}}

\newcommand{\fa}{\mathfrak{a}}
\newcommand{\fb}{\mathfrak{b}}

\newcommand{\CL}{\mathrm{CL}}    

\newtheorem{theorem}{Theorem}[section] 

\newtheorem{proposition}[theorem]{Proposition}
\newtheorem{Conjecture}[theorem]{Conjecture} 
\newtheorem{remark}[theorem]{Remark} 
\newtheorem{lemma}[theorem]{Lemma}

\begin{document}
\title{{On imaginary quadratic fields with non-cyclic class groups}}
\author{Yi Ouyang$^{1,2}$ , Qimin Song$^1$ and Chenhao Zhang$^1$}
\address{$^1$School of Mathematical Sciences, Wu Wen-Tsun Key Laboratory of Mathematics,   University of Science and Technology of China, Hefei 230026, Anhui, China}
	
\address{$^2$Hefei National Laboratory, University of Science and Technology of China, Hefei 230088, China}
	
	\email{yiouyang@ustc.edu.cn}	
	\email{sqm2020@mail.ustc.edu.cn}
	\email{chhzh@mail.ustc.edu.cn}
	
\thanks{Partially supported by NSFC (Grant No. 12371013) and  Innovation Program for Quantum Science and Technology (Grant No. 2021ZD0302902).}
	

\begin{abstract} For a fixed abelian group $H$, let $N_H(X)$ be the number of square-free positive integers $d\leq X$ such that $H\hookrightarrow \CL(\Q(\sqrt{-d}))$. We obtain asymptotic lower bounds for $N_H(X)$ as $X\to\infty$ in two cases:  $H=\Z/g_1\Z\times (\Z/2\Z)^l$ for $l\geq 2$ and $2\nmid g_1\geq 3$, $H=(\Z/g\Z)^2$ for $2\nmid g\geq 5$. More precisely, for any $\epsilon >0$, we showed $N_H(X)\gg X^{\frac{1}{2}+\frac{3}{2g_1+2}-\epsilon}$ when $H=\Z/g_1\Z\times (\Z/2\Z)^l$ for $l\geq 2$ and $2\nmid g_1\geq 3$. For the second case, under a well known conjecture for square-free density of integral multivariate polynomials, for any $\epsilon >0$, we showed $N_H(X)\gg X^{\frac{1}{g-1}-\epsilon}$ when $H=(\Z/g\Z)^2$ for $ g\geq 5$. The first case is an adaptation of Soundararajan's results for $H=\Z/g\Z$, and the second conditionally improves the bound $X^{\frac{1}{g}-\epsilon}$ due to Byeon and the bound $X^{\frac{1}{g}}/(\log X)^{2}$ due to Kulkarni and Levin.
\end{abstract}
\maketitle	

\section{Introduction}
In this note we assume $d>1$ is a square-free integer if there is no further notice. Let $\CL(-d)$  be the class group of  the imaginary quadratic field $\Q(\sqrt{-d})$.	Let $H$ be a fixed abelian group. For $X>0$,  let
 \begin{equation}
 	N_H(X)= \ \#\{d\mid d\leq X,\ \exists\ \text{inclusion}\  H \hookrightarrow \CL(-d)\}.
 \end{equation}
There is a lot of interest to study the density/asymptotic behavior of $N_H(X)$ as $X\to\infty$ for different $H$.  

 Set
 \[ H_1=\Z/g\Z, \quad H_2=(\Z/2\Z)^l \times \Z/g_1\Z,\quad H_3=\Z/g\Z\times \Z/g\Z.    \]    
In the literature, $N_{H_1}(X)$ is denoted as $N_g(X)$ and $N_{H_3}(X)$  as $N^-(g^2;X)$. It is believed that
 \[ N_{g}(X)\sim C_{g}X \]
where $C_g$ is a positive constant depending only on $g$. In particular, when $g$ is an odd prime, H. Cohen and H. W. Lenstra~\cite{ref4} conjectured that
 \[ C_{g}=\frac{6}{{\pi}^{2}}\left(1-\prod_{j=1}^{\infty}(1-g^{-j})\right). \] 
There are similar conjectures for  $N^{-}(g^{2};X)$.

Ankeny and Chowla~\cite{ref12} showed that $N_{g}(X)$ tends to infinity with X , and in fact their method imply $N_{g}(X)\gg X^{1/2}$. Murty~\cite{ref2} proved that $N_{g}(X)\gg X^{\frac{1}{2}+\frac{1}{g}-\epsilon}$ when $X\to \infty$. Soundararajan~\cite{ref1} improved Murty's bound and showed $N_{g}(X)\gg X^{\frac{1}{2}+\frac{2}{g}-\epsilon}$ for $g\equiv 0 \mod 4$, and $N_{g}(X)\gg X^{\frac{1}{2}+\frac{3}{g+2}-\epsilon}$ for $g\equiv 2 \mod 4$. Noted that $N_{g}(X)\geqslant N_{2g}(X)$, Soundararajan's result  contains the bound for $N_{g}(X)$ when $g$ is odd. For $g=3$, Health-Brown~\cite{ref3} showed that $N_{3}(X)\gg X^{\frac{9}{10}-\epsilon}$.  Byeon~\cite{ref6} showed that $N^{-}(g^{2};X)\gg X^{\frac{1}{g}-\epsilon}$ for odd integers $g$. For $g=3$,  Yu~\cite{ref5} proved that $N^{-}(3^{2};X)\gg X^{\frac{1}{2}-\epsilon}$.  Kulkarni and Levin~\cite{ref11} showed that $N^{-}(g^{2};X)\gg X^{\frac{1}{g}}/(\log X)^{2}$ for any integer $g$.

In this note we make a further study of the asymptotic lower  bounds for $N_{H_2}(X)$  and $N_{H_3}(X)$  as $X\to \infty$.  Based on Soundararajan's method, we get a  bound for $N_{H_2}(X)$ when $l\geq 2$ and $g_1\geq 3$ (Theorem~\ref{thm:11}). For $g\geq 5$,	we construct a family of imaginary quadratic fields whose ideal class groups have subgroups isomorphic to $\Z/g\Z\times \Z/g\Z$. Based on this construction, we obtain a bound for $N_{H_3}(X)$ for $ g\geq 5$  under a well known conjecture for square-free density of integral multivariate polynomials (Theorem~\ref{thm:12}).\\
\textbf{Notation}: $\mu (n)$ stand for the M\"{o}bius function; $f(x)\ll g(x)$ or $f(x)=O(g(x))$ means that there is a constant $c>0$ such that $|f(x)|\leqslant cg(x)$; $f(x)\asymp g(x)$ means that $f(x)=O(g(x))$ and $g(x)=O(f(x))$; $f\sim g$ means that $\lim_{x\rightarrow \infty} \frac{f(x)}{g(x)}=1$.

\section{Bound for $N_{H_2}(X)$}
		
\begin{theorem}\label{thm:11} Suppose $l\geqslant 2$ and $2\nmid g_{1}\geqslant 3$. Then as $X\to \infty$, we have
\begin{equation*}
		N_{H_2}(X)= \#\{ d\leq X:\ \exists\ (\Z/2\Z)^l \times \Z/g_1\Z \hookrightarrow \CL(-d)\} \gg X^{\frac{1}{2}+\frac{3}{2g_1+2}-\epsilon}
\end{equation*}
for any $\epsilon >0$.
\end{theorem}	
\begin{proof} Take $l$ different odd primes $p_{1},..., p_{l}$ ($p_{i}>3$). For each $i$, we choose integers $a_{i}$ and $b_{i}$ such that $p_{i}\nmid 2a_{i}-g_{1}b_{i}$. Let $n_{i}=1+a_{i}p_{i}$ and $m_{i}=1+b_{i}p_{i}$. Then one can see that $p_{i}\mid n_{i}^{2}-m_{i}^{g_{1}}$ and $p_{i}^{2}\nmid n_{i}^{2}-m_{i}^{g_{1}}$. Let $n,\ m$ be solutions of the congruence equations
 \begin{equation} \label{eq:nm}
  \begin{cases}(n,m) \equiv (2,1) \mod 18,\\ (n,m)\equiv (n_{i},m_i) \mod p_{i}^{2}, i=1,...,l.
 \end{cases}  \end{equation}
By Chinese Remainder Theorem, $n$ and $m$ belong to two different congruence classes modulo $18\prod_{i=1}^{l}p_{i}^{2}$.	
	
Let  $T\leqslant \frac{X^{1/2}}{64}$ be a parameter to be chosen later. Set $M=\dfrac{T^{\frac{2}{g_1}}X^{\frac{1}{g_1}}}{2}$ and $N=\dfrac{TX^{\frac{1}{2}}}{2^{g_{1}+1}}$. 
For $d\leqslant X$, if $d$ is not squarefree, let $R(d)=0$; if $d$ is square-free, let $R(d)$ be the number of solutions $(m,n,t)$ of the equation $m^{g_{1}}=n^{2}+t^{2}d$, subject to \eqref{eq:nm} and the conditions
			\begin{equation*}
				t\nmid m, M<m\leqslant 2M, N<n \leqslant 2N, T<t \leqslant 2T.
			\end{equation*}
If $R(d) >0$, on the one hand, $\CL(-d)$ has an element of order $g_{1}$ by \cite[Proposition 1]{ref1}. On the other hand, we have $p_i\mid d$ for $1\leq i\leq l$ and $3|d$, hence $\CL(-d)$ has a subgroup isomorphic to $(\Z/2\Z)^{l}$ by Gauss's genus theory, thus $\CL(-d)$ contains $H_2=(\Z/2\Z)^l \times \Z/g_1\Z$ as a subgroup.

Let	$S_{1}=\sum_{d\leqslant X} R(d)$, $S_{2}=\sum_{d\leqslant X} R(d)(R(d)-1)$. By Cauchy's inequality, we have:
		    \begin{equation}\label{eq1}
		    N_{H_2}(X)\geqslant	\#~\{d\leqslant X:R(d)\neq 0 \}\geqslant \frac{S_{1}^{2}}{S_{1}+S_{2}}.
		    \end{equation}
Now changing the congruence conditions and following  Soundararajan's method of estimating $(1.4)$ in ~\cite{ref1}  we get
		    \begin{equation*}
		    	S_{1}\asymp \frac{MN}{T}  
		    \end{equation*}
By the same argument as the estimate of $(1.5)$ in ~\cite{ref1}, we get $S_{2}\ll T^{2}M^{2}X^{\epsilon}$.\\
Take $T=X^{\frac{g_{1}-2}{4g_{1}+4}}$ in \eqref{eq1} we get the desired bound.
\end{proof}
\begin{remark} Suppose $g_1=3$. We can apply Health-Brown's estimate in~\cite{ref3} for $S_{2}$ to get  $N_{H_2}(X) \gg X^{\frac{9}{10}-\epsilon}$ for any $\epsilon >0$.  We can also use a criterion of Honda~\cite[Proposition 10]{ref8} and combine the above construction to get $N^{+}_{H_2}(X):= \#\{0<d\leqslant X: H_2\hookrightarrow\ \CL(\Q(\sqrt{d}))\} \gg X^{\frac{9}{10}-\epsilon}$ for any $\epsilon >0$.
\end{remark}

\section{Bound for $N_{H_3}(X)$ }
The following proposition is an extension of ~\cite[Lemma 2.1]{ref5} and ~\cite[Proposition 1]{ref10}.
\begin{proposition}\label{prop:31}
Let $g\geqslant 3$ be an integer. For positive integers $a,b,n$, denote $ f_1(a,b)=\sum\limits_{i=0}^{g-1}a^{g-1-i}b^{i}$ and 
 \[  f(a,b,n)=2(a^{g}+b^{g})n^{g}-(a-b)^{2}n^{2g}-f_1(a,b)^{2}. \]
Let $f(a,b,n)= D$. Suppose $ab>1$. If $D\geqslant 4\max\{ba^{g-2}n^{g-1}, ab^{g-2}n^{g-1} \}$ and is square-free, then $\CL(-D)$ contains a subgroup isomorphic to $\Z/g\Z \times \Z/g\Z$. 
\end{proposition}
\begin{proof}Let $X_{1}=f_1(a,b)+(a-b)n^{g}$, $Y_{1}=an$, $X_{2}=f_1(a,b)-(a-b)n^{g}$, $Y_{2}=bn$. We have
		\begin{equation*}
			X_{1}^{2}-4Y_{1}^{g}=X_{2}^{2}-4Y_{2}^{g}=-f(a,b,n)=-D. 
		\end{equation*}
Thus $(\frac{X_{i}+\sqrt{-D}}{2})(\frac{X_{i}-\sqrt{-D}}{2})=Y_{i}^{g}$, which implies that
$X_{i}$ and $Y_{i}^{g}$ are elements in the ideal $ (\frac{X_{i}+\sqrt{-D}}{2},\frac{X_{i}-\sqrt{-D}}{2})$.
Since $D$ is squarefree,  we have $(X_{i},Y_{i})=(X_{j},Y_{j})=1$, and the ideals $(\frac{X_{i}+\sqrt{-D}}{2})$ and $(\frac{X_{i}-\sqrt{-D}}{2})$ are coprime. Hence we can write $(\frac{X_{i}+\sqrt{-D}}{2})=\fa_{i}^{g}$ and $\frac{X_{i}-\sqrt{-D}}{2}=\bar{\fa}_{i}^{g}$ for some integral ideals $\fa_{i}, i=1,2$. We show that $[\fa_{i}]\ (i-1,2)$ , and either  $[\fa_{1}\fa_{2}^{k}]\ (1\leq k\leq g-1)$ or $[\fa_{2}\fa_{1}^{k}]\ (1\leq k\leq g-1)$ are all elements of order $g$ in $\CL(-D)$, consequently  $\langle [\fa_1],[\fa_2]\rangle$ is a subgroup of $\CL(-D)$ isomorphic to $\Z/g\Z \times \Z/g\Z$.

\medskip \noindent
(1) For $i=1,2$, we show that $[\fa_{i}]$ is an element of order $g$. If not, then $\fa_{i}$ is an element of order $r<g$. Write $\fa_{i}^{r}=(\frac{\alpha+\beta\sqrt{D}}{2})$. Note that $\beta \neq 0$,  we have
	\begin{equation*}
				 Y_{i}^{g}=N(\fa_{i}^{g})=(N(\fa_{i}^{r}))^{\frac{g}{r}}=\left(\frac{\alpha^{2}+D\beta^{2}}{4}\right)^{\frac{g}{r}}\geqslant \left(\frac{1+D^{2}}{4}\right)^{3}>Y_{i}^{g},
	\end{equation*} 
which is a contradiction.
			
\medskip	\noindent		
(2)  We claim that  at least one of the following two conclusions is true:
\begin{enumerate}[i]
	\item $[\fa_{1}\fa_{2}^{k}]\ (1\leq k\leq g-1)$ are elements of order $g$ in $\CL(-D)$.
	\item $[\fa_{2}\fa_{1}^{k}]\ (1\leq k\leq g-1)$ are elements of order $g$ in $\CL(-D)$.
\end{enumerate} 

Assume both conclusions are false. 

(2-1) Since (i) is not true, there exists some $s\leq g-1$ such that $[\fa_{1}\fa_{2}^{s}]$  is an element of order $r<g$. Hence  $\fa_{1}^{r}\fa_{2}^{sr}$ is principal. Write $sr=kg+s_{1}, 0\leqslant s_{1}<g$, thus $\fa_{1}^{r}\fa_{2}^{s_{1}}$ is principal. Replacing $\fa_2^{s_{1}}$ by $\bar{\fa}_{2}^{g-s_{1}}$ if necessary, we may assume  $s_{1}<\dfrac{g}{2}$, thus we have $\fa_{1}^{r}\fb^{s_{1}}$ is principal for $\fb=$ $\fa_{2}$ or $\bar{\fa}_{2}$. Denote $\fa_{1}^{r}\fb^{s_{1}}=(\frac{\alpha+\beta\sqrt{D}}{2})$ for some integers $\alpha,\beta$ of same parity. Note that $r\mid s_{1}$, denote $s_{1}=tr$, we have
 \[ \left(\fa_{1}^{r}\fb^{s_{1}}\right)^{\frac{g}{r}}=
 \left(\frac{X_{1}+\sqrt{-D}}{2}\right)\left(\frac{X_{2}\pm\sqrt{-D}}{2}\right)^{t}.\] 

If $t$ is even, say $t=2t_{1}$, denoted $(\frac{X_{1}+\sqrt{-D}}{2})(\frac{X_{2}\pm\sqrt{-D}}{2})^{t}=(\frac{A+B\sqrt{-D}}{2})$, thus we have
	\begin{equation*}
		2^{t}B=\sum_{i=0}^{t_{1}} \binom{2t_{1}}{2i}X_{2}^{2t_{1}-2i}(-D)^{i}\pm X_{1}\sum_{i=0}^{t_{1}-1}\binom{2t_{1}}{2i+1}X_{2}^{2t_{1}-2i-1}(-D)^{i}.
	\end{equation*}
	If $B=0$, then $X_{2}\mid D^{t_{1}}$, contradiction to $(X_{2},D)=1$. Hence $B\neq 0$, which implies $\beta \neq 0$.
	Note that $r+s_{1}\leqslant g-1$, we get 
	\begin{equation}\label{eq:3}
		Y_{1}^{r}Y_{2}^{s_{1}}=N(\left(\fa_{1}^{r}\fb^{s_{1}}\right))\geqslant \left(\frac{1+D}{4}\right)>\max\{Y_{1}^{g-2}Y_{2}, Y_{2}^{g-2}Y_{1} \}>Y_{1}^{r}Y_{2}^{s_{1}}
	\end{equation}  
	which is a contradiction.
	
Now  $t$ must be odd, say $t=2{t_{2}+1}$, denote $(\frac{X_{1}+\sqrt{-D}}{2})(\frac{X_{2}\pm\sqrt{-D}}{2})^{t}=(\frac{A_{1}+B_{1}\sqrt{-D}}{2})$, thus we have
\begin{equation*}
	2^{t}B_{1}= \sum_{i=0}^{t_{2}}\binom{2t_{2}+1}{2i}(-D)^{i}X_{2}^{2t_{2}+1-2i}\pm X_{1}\sum_{i=0}^{t_{2}}\binom{2t_{2}+1}{2i+1}(-D)^{i}X_{2}^{2t_{2}-2i}.
\end{equation*}
If $B_{1}\neq 0$, then $\beta\neq 0$, we still get a contradiction by \eqref{eq:3}. Thus we have $B_{1}=0$, which implies $X_{2}\mid X_{1}D^{t_{2}}$ and then $X_{2}\mid X_{1}$ as $(X_{2},D)=1$.

(2-2) Since (ii) is also not true, by a symmetric argument we get $X_{1}\mid X_{2}$.
 
By (2-1) and (2-2)  we have $|X_{1}|=|X_{2}|$, which implies $a=b$. Since $D$ is square-free and $ab\neq 1$, we have $a\neq b$, which is a contradiction.
\end{proof} 
	
\begin{lemma}\label{lemma:32} Let $f_{1}(x,y)=(\sum\limits_{i=0}^{g-1} x^{g-1-i}y^i)^{2}$ and
 \[f(x,y,z)=2(x^{g}+y^{g})z^{g}-(x-y)^{2}z^{2g}-(f_{1}(x,y))^{2} \]  
as given in Proposition~\ref{prop:31}. Then $f(x,y,z)$ is square-free  in $\Z[x,y,z]$.
\end{lemma}
\begin{proof}
If there exist $h(x,y,z), k(x,y,z) \in \Z[x,y,z]$ such that  
		\begin{equation}\label{eq:4}
			f(x,y,z)=h(x,y,z)k(x,y,z)^{2},
		\end{equation}
then
		\begin{equation}\label{eq:5}
			k(x,y,z)\mid \frac{\partial}{\partial z}f(x,y,z).
		\end{equation}
		
Let $h_1(x,y)=h(x,y,0)$, $k_1(x,y)=k(x,y,0)$. We have  $k_{1}(x,y)\mid f_{1}(x,y)^{2}$ by \eqref{eq:4} and $k_{1}(x,y)\mid 2g(x^{g}+y^{g})$ by \eqref{eq:5}. Then $k_{1}(x,y)=\pm 1$ since $f_{1}(x,y)$ is coprime with $2g(x^{g}+y^{g})$. Thus we have $h_{1}(x,y)=\pm (f_{1}(x,y))^{2}$. Considering the total degree in \eqref{eq:4}, note that $f(x,y,z)$ is a polynomial of  total degree $2g+2$ and $\deg h \geqslant \deg h_{1} =2\deg f_{1}= 2g-2$, we have $\deg k \leqslant 2$. Denote $k(x,y,z)=az^{2}+h(x,y)z\pm 1$ for $a\in \Z$, where $h(x,y)$ is an integral polynomial of total degree $\leq 1$. Compare the degrees of $x$ and $y$ in \eqref{eq:4}, we have $h(x,y)\in \Z$, and hence $k(x,y,z)\in \Z[z]$. Take $x=y$ in \eqref{eq:4}, we have $k(x,y,z)^{2}\mid 4z^{g}-g^{2}y^{g-2}$, thus $k(x,y,z)=\pm 1$.
\end{proof}	 

Now we recall the conjecture for squarefree density of integral multivariate polynomials (see ~\cite{ref9,ref7}  for more details). Suppose $P$ is a polynomial in $\Z[X_{1}, X_{2}, \cdots, X_{n}]$  of total degree $d\geqslant 2$. For any integer $m>1$, let
	\begin{equation*}
		\rho_{P}(m)=\# \{X\in (\Z/m\Z)^{n}: P(X)\equiv 0 \mod m\}.
	\end{equation*}
Given $B_{j}\in \R$, $B_{j}\geqslant 1\ (j=1,...,n)$ and $h\in \Z$, define
	\begin{equation*}
		B= \prod_{j=1}^n [0,B_{j}]\cap\Z^n,\quad  r_{P}(h)=\# \{X\in B\mid P(X)=h \},
	\end{equation*}
	\begin{equation*}
		N_{P}(B)=\sum_{h\in\Z,\ h\neq 0}\mu(|h|)^{2}r_{P}(h). 
	\end{equation*}
Note that $N_P(B)$ is	 the number of $X\in B$ such that $P(X)$ takes square-free value. 

\begin{Conjecture}\label{conj:33}
$N_{P}(B)\sim \mathscr{C}_{P}B_{1}...B_{n}$ as $ \min\limits_{j=1,\cdots, n}B_{j}\rightarrow \infty$, 
where 
	\begin{equation*}
	 \mathscr{C}_{P}=\prod_{p}\left(1-\frac{\rho_{P}(p^{2})}{p^{2n}}\right). 
	 \end{equation*}
\end{Conjecture}

\begin{theorem}\label{thm:12}
Assume Conjecture~\ref{conj:33} holds. Then for $g\geq 5$, 
	\begin{equation*}
		N_{H_3}(X)= \#\{ d\leq X:\ \exists\ (\Z/g\Z)^2  \hookrightarrow \CL(-d)\}\gg X^{\frac{1}{g-1}-\epsilon} 
	\end{equation*}
	for large $X$ and any $\epsilon >0$.
\end{theorem}	
\begin{proof}
For a given large $X$, let
 \begin{equation*}
 R =\left(\Bigl(\frac{1}{2^{4}g^{2}}X^{\frac{1}{2(g-1)}},\frac{2^{2g-4}g^{g-2}+1}{2^{2g}g^{g}}X^{\frac{1}{2(g-1)}}\Bigr)\cap \Z\right)^{2}\times \left(\Bigl(\frac{1}{2}X^{\frac{g-2}{2g(g-1)}},X^{\frac{g-2}{2g(g-1)}}\Bigr)\cap \Z\right). 
 \end{equation*} 
For $(x,y,z)\in R$, let $f(x,y,z)$ be the polynomial defined in Lemma~\ref{lemma:32}. we have $f(x,y,z)>0$ and $f(x,y,z)\asymp X$. Let 
	 \begin{equation*}
	 		r(D):=\#\{(x,y,z)\in R:\ D=f(x,y,z)\}.
	 \end{equation*}
By repeatedly using Conjecture~\ref{conj:33} and the inclusion-exclusion principle we get
	\begin{equation*}
	 		S_{1}:=\sum_{D}\mu(|D|)^{2}r(D)\sim c\mathscr{C}_{f}X^{\frac{3g-2}{2g(g-1)}}
	\end{equation*} 
for some constant $c>0$. By Lemma~\ref{lemma:32} and Theorem 1.1 in ~\cite{ref7} , we get $\mathscr{C}_{f}>0$, thus $S_{1}\asymp X^{\frac{3g-2}{2g(g-1)}}$.
We will see 
	 \begin{equation}\label{eq:6}
	 		S_{2}=\sum_{D}\mu(|D|)^{2}r(D)^{2}\ll X^{\frac{2}{g}+\epsilon}.
	 \end{equation}
Then by Proposition~\ref{prop:31} and Cauchy's inequality, we have:
	 	\begin{equation*}
	 		N_{H_3}(X)\geqslant \sum_{{\tiny \begin{aligned} &D\leqslant X \\ r&(D)>0	\end{aligned}}}\mu(D)^{2}\geqslant \frac{S_{1}^{2}}{S_{2}}\gg X^{\frac{1}{g-1}-\epsilon}.
	 	\end{equation*}
To show \eqref{eq:6}, note that $S_{2}$ is bounded by the number of solutions to the equation
 \begin{equation*}
 4z_{1}^{g}x_{1}^{g}-4z_{2}^{g}x_{2}^{g}=\big(\sum_{i=0}^{g-1}x_{1}^{g-1-i}y_{1}^{i}+(x_{1}-y_{1})z_{1}^{g}\big)^{2}-\big(\sum_{i=0}^{g-1}x_{2}^{g-1-i}y_{2}^{i}+(x_{2}-y_{2})z_{2}^{g}\big)^{2}
 \end{equation*}
subject to $(x_{i},y_{i},z_{i})\in R$ for $i=1,2$.

Fix $z_{1}, x_{1}$. If $z_{1}^{g}x_{1}^{g}=z_{2}^{g}x_{2}^{g}$, then there are $O(z_{1}x_{1})=O(X^{\epsilon})$ choices of $x_{2},z_{2}$, and  for each choice of $y_{1}$, there are at most $2g-2$ choices of $y_{2}$, thus there are at most $O(X^{\frac{3g-2}{2g(g-1)}+\epsilon})$ solutions. Now for any choice of $x_{1},z_{1},x_{2},z_{2}$ such that $z_{1}^{g}x_{1}^{g}\neq z_{2}^{g}x_{2}^{g}$, there are at most $O(|z_{1}^{g}x_{1}^{g}- z_{2}^{g}x_{2}^{g}|)=O(X^{\epsilon})$ choices of integers $s$ and $t$ such that:
	 	\begin{equation*}
	 		\sum_{i=0}^{g-1}x_{1}^{g-1-i}y_{1}^{i}+(x_{1}-y_{1})z_{1}^{g}=s, \quad \sum_{i=0}^{g-1}x_{2}^{g-1-i}y_{2}^{i}+(x_{2}-y_{2})z_{2}^{g}=t.
	 	\end{equation*}
For each $s,t$, there are at most $g-1$ solutions of $y_{1},y_{2}$, thus there are at most $O(X^{\frac{2}{g}+\epsilon})$ solutions. Thus we get the estimate in \eqref{eq:6}.

\end{proof}

\end{document}